\documentclass[a4paper,12pt]{article}
\usepackage{amsmath,amssymb,amsxtra,amsthm}

\theoremstyle{plain}
\newtheorem{theorem}{Theorem}[section]
\newtheorem{lemma}[theorem]{Lemma}

\newtheorem{assumption}[theorem]{Assumption}
\theoremstyle{definition}

\makeatletter
 \def\address#1#2{\begingroup
 \noindent\parbox[t]{7.8cm}{%
 \small{\scshape\ignorespaces#1}\par\vskip0ex
 \noindent\small{\itshape E-mail}%
 \/: #2\par\vskip4ex}\hfill%
 \endgroup}%
\makeatother
    
\newcommand{\re}{\mathbb{R}}
\def\bra#1{\langle {#1} \rangle }
\def\r2n{\re^n\times (\re^n\!\setminus\!\{ 0\})}
\def\dbar{\mbox{\setbox0=\hbox{$d$}$d$\kern-.55\wd0\vbox{%
\hrule height.1ex width.75\wd0\kern1.3ex}}}

\date{}
\begin{document}
\title{Remark on Estimates in Modulation Spaces for Schr\"odinger Evolution
Operators with Sub-quadratic Potentials
}
\author{
Kosuzu Hamaoka, Keiichi Kato and Shun Takizawa
}
\maketitle

\begin{abstract}
In this paper we give an estimate for the solution to the
 Schr\"odinger equation with sub-quadratic potentials
 in modulation spaces by the norm of the initial functions in Wiener-Amalgum spaces. 
\end{abstract}

\section{Introduction}
In this paper, we shall give an estimate for the solution to the Schr\"odinger equation with time
 dependent potential:
    \begin{align}
    \begin{cases}
    i\partial_t u(t,x)=-\frac{1}{2}\Delta u(t,x)+V(t,x)u(t,x),  &(t,x)\in
    \re\times\re^n,\\
    u(0,x)=u_0(x), &x\in \re^n
    \end{cases}
   \label{SE}
    \end{align}
 in modulation spaces by the norm of the initial function in Wiener-Amalgum spaces. 
Here $i= \sqrt{-1}$,  $u(t,x)$ is a complex valued function of 
 $(t,x) \in {\mathbb R} \times {\mathbb R}^n$, $V(t,x)$ is a real valued
 function of $(t,x) \in {\mathbb R} \times {\mathbb R}^n$,
 $u_0(x)$ is a complex valued function of $x \in {\mathbb R}^n$,
 $\partial_t u= {\partial u}/{\partial t}$ and
 $\Delta u= \sum^{n}_{i=1} {\partial^2 u }/{ \partial x_i^2}$.

It is shown that the solutions in modulation spaces $M^{p,q}$ with $q\ge p$ are estimated by the Wiener-Amalgum spaces $W^{p,q}$ with the same index 
under the following condition on the potential $V$. 
\begin{assumption} 
\label{ass1-1}
 $V \in C^{\infty}({\mathbb R} \times {\mathbb R}^n)$ and there exists a positive number $\epsilon $ such that 
 for all multi-indices $\alpha$
 there exists $C_\alpha >0$ such that 
    \begin{align}
    \label{assumption_V}
    |\partial^\alpha_xV(t,x)|\le C_{\alpha}(1+|x|)^{2-\epsilon -|\alpha |},
    \quad (t,x)\in \re\times\re^n .
    \end{align}
\end{assumption}
There are a large number of works devoted to study the equation \eqref{SE}.
The works of 
 B\'{e}nyi-Gr\"{o}chenig-Okoudjou-Rogers \cite{Benyi et all},
 Wang-Hudzik \cite{Wang-Hudzik} and Wang-Zhao-Guo \cite{WZG} are the beginnings In the context of the modulation spaces $M^{p,q}$, and these 
 works have been developed by a number of authors using a large variety of 
 methods (see, for example,  B\'{e}nyi-Okoudjou \cite{Benyi-Okoudjou},
 Cordero-Nicola \cite{Cordero Nicola JFA 2008}, 
\cite{Cordero Nicola MN 2008}, Cordero-Nicola-Rodino \cite{Cordero Nicola Rodino}, 
 Kobayashi-Sugimoto \cite{KS}, Miyachi-Nicola-Rivetti-Tabacco-Tomita
 \cite{Miyachi-NicolaRivetti-Tabacco-Tomita}, Tomita \cite{Tomita}, 
 Wang-Huang \cite{Wang Huang}). 
The precise definition of modulation spaces and Wiener-amalgum spaces will be given in Section
 \ref{Preliminaries}, but the main idea of these spaces is to consider
 the space variable and the variable of its Fourier transform 
 simultaneously.

To state our results, we define the Schr\"odinger operator of a free
 particle  $e^{\frac{1}{2}it\Delta}$ by 
    $$(e^{\frac{1}{2}it\Delta}f)(x)=
    \mathcal{F}^{-1}_{\xi\to x}[e^{-\frac{1}{2}it|\xi|^2}
    \mathcal{F}f(\xi)](x), \quad f\in\mathcal{S}(\re^n).$$ 
Here we use the notation
 $\mathcal{F}f(\xi)=\int_{\re^n}f(x)e^{-ix\cdot\xi}dx$ for the Fourier
 transform of $f$ and 
 $\mathcal{F}^{-1}f(x)=\int_{\re^n} f(\xi)e^{ix\cdot\xi}\dbar\xi$ with
 $\dbar\xi =(2\pi)^{-n}d\xi$ for the inverse Fourier transform of $f$.
The following theorems are our main results.

\begin{theorem}
\label{main-estimate}
Suppose that $1\le p, q \le \infty$ satisfy 
\begin{equation}
 \begin{cases}
  p\le q \quad \text{for} \ p \ge n/\epsilon, \\
p \le q < \frac{np}{n-p\epsilon} \quad \text{for} \ p<n/\epsilon.
\end{cases}
\end{equation}
Let $\varphi_0\in\mathcal{S}(\re^n)\backslash\{ 0\}$ and 
set $\varphi (t,x)=e^{\frac{1}{2}it\Delta} \varphi_0(x)$.
If $V$ satisfies Assumption \ref{ass1-1}, 
then 
 there exist $T>0$ and $C>0$ such that 
    \begin{align*}
    \| u(t,\cdot) \|_{M^{p,q}_{\varphi (t,\cdot)}}
    \le C \|u_0\|_{W^{p,q}_{\varphi_0}},
    \quad u_0\in\mathcal{S}(\re^n)
    \end{align*}
 for all $t\in [-T,T]$, where $u(t,x)$ denotes the solution of
 \eqref{SE} in $C(\re ;\mathcal{S}'(\re^n))$
 with $u(0,x)=u_0(x)$.
\end{theorem}

Cordero-Gr\"ochenig-Nicola-Rodino \cite{Cordero-Grochenig-Nicola-Rodino} and 
Kato-Ito-Kobayashi \cite{K-K-I-4} have shown that 
the solutions of Schr\"odinger equations with at most quadratic growth potentials in modulation spaces $M^{p.p}$ is estimated by the same norm of the initial data. On the other hand, if the potential $V$ has at most linear growth, Kato-Ito-Kobayashi \cite{K-K-I-4} has shown that the solutions  in modulation spaces $M^{p,q} $ is estimated 
by the same norm of the initial functions for any $p,q $ with $1\le p,q \le \infty$ .
Theorem \ref{main-estimate} shows that 
the solutions in modulation spaces $M^{p,q}$ can be estimated by 
the Wiener-amalgum space with the same parameter for short time, 
which is impossible if the potential has quadratic growth.

This paper is organized as follows.
In Section 2, we give some notation and recall the definitions and 
 basic properties of wave packet transform and modulation spaces.
In Section 3, we give some properties concerning the orbit of the classical
 mechanics corresponding to the Schr\"odinger equation \eqref{SE}.
In Section 4, we prove Theorem \ref{main-estimate}.

\section{Preliminaries}\label{Preliminaries}

\subsection{Notation}
For $x=(x_1,\cdots ,x_n)\in\re^n$, we denote 
$\bra{x}=(1+|x|^2)^{\frac{1}{2}}$.
For a real valued function $V\in C^1(\re\times \re^n)$,
we put
$$\nabla_x V(t,x_1,\ldots ,x_n)=(\partial_{x_1}V(t,x_1,\ldots ,x_n)
 ,\ldots , \partial_{x_n}V(t,x_1,\ldots ,x_n) ).$$
Throughout this paper the letter $C$
denotes a constant, which may be different in each occasion.

\subsection{Wave Packet Transform}\label{Wave Packet Transform}
We recall the definition of the wave packet transform which is defined
 by C\'ordoba-Fefferman \cite{C-F-1}. 
Wave packet transform is called short time Fourier transform or windowed
 Fourier transform in several literatures.
Let $f \in {\mathcal S}^\prime({\mathbb R}^n)$ and 
 $\varphi \in {\mathcal S}({\mathbb R}^n)\backslash \{0\}$.
Then the wave packet transform $W_\varphi f(x,\xi)$ of $f$ with 
 the wave packet generated by a function $\varphi$ is defined by
    \begin{equation*}
    W_\varphi f(x,\xi) = \int_{{\mathbb R}^n} 
    \overline{ \varphi(y-x)} f(y) e^{-i y \cdot \xi} dy.
    \end{equation*} 
We call such $\varphi$ window function.
Let $F$ be a  function  on ${\mathbb R}^n \times {\mathbb R}^n$.
Then the (formal) adjoint operator $W^{*}_\varphi$ of $W_\varphi$ is
 defined by
    $$W^{*}_\varphi F(x)=\iint_{{\mathbb R}^{2n}} F(y,\xi)
    \varphi(x-y)  e^{ix \cdot \xi} dy \dbar\xi$$
with $\dbar \xi = (2 \pi)^{-n} d \xi$.
It is known that  for 
 $\varphi, \psi \in {\mathcal S}({\mathbb R}^n)\backslash\{ 0\}$
 satisfying 
$(\psi, \varphi)_{L^2}\not=0$
we have the inversion formula
    \begin{equation*}
\frac{1}{(\psi, \varphi)_{L^2}}W_\psi^{*} W_\varphi f = f,
    \quad f \in {\mathcal S}^\prime({\mathbb R}^n) \label{inversion formula}
    \end{equation*}
 ([\ref{Grochenig}, Corollary 11.2.7]).

For the sake of convenience, we use the following notation
    \begin{align*}
    W_{\varphi(t,\cdot)}u(t,x,\xi)
    =\int_{\re^n}\overline{\varphi (t,y-x)}u(t,y)e^{-iy\cdot \xi}dy,
    \end{align*}
 where  
 $\varphi (t,x)$ and $u(t,x)$ are functions on  $\re\times\re^n$.

\subsection{Modulation Spaces and Wiener-Amalgum spaces}\label{MODULATION SPACES}
We recall the definitions of modulation spaces $M^{p,q}$ and Wiener-amalgum spaces $W^{p,q}$.
Let $1 \leq p,q \leq \infty$ and 
 $\varphi \in {\mathcal S}({\mathbb R}^n) \setminus \{ 0 \}$.
Then  the modulation spaces $M^{p,q}_\varphi({\mathbb R}^n)=M^{p,q}$
(the Wiener-amalgum spaces $W^{p,q}_\varphi({\mathbb R}^n)=W^{p,q}$)
 consists of all tempered distributions 
 $f \in {\mathcal S}^\prime({\mathbb R}^n)$ such that the norm
\begin{equation*}
 \| f \|_{M^{p,q}_\varphi} =
    \left( \int_{\re^n} \left( \int_{\re^n} |W_\varphi  f (x,\xi)|^p  dx  \right)^{q/p} d \xi  \right)^{1/q}
= \| W_\varphi f(x,\xi) \|_{L^p_x L^q_\xi}
\end{equation*}
\begin{equation*}
\left(\text{resp.} \quad
 \| f \|_{W^{p,q}_\varphi} = \left( \int_{{\mathbb R}^n} \left( \int_{{\mathbb R}^n} |W_\varphi  f (x,\xi)|^q  d\xi  \right)^{p/q} dx  \right)^{1/p}
= \| W_\varphi f(x,\xi) \|_{L^q_\xi L^p_x } 
\right)
\end{equation*}
 is finite(with usual modifications  if $p=\infty$ or $q=\infty$).

The space $M^{p,q}_\varphi({\mathbb R}^n)$ and $W^{p,q}_\varphi({\mathbb R}^n)$ are Banach spaces,
whose definitions are  independent of the choice of 
the window function $\varphi$, i.e.,
$M^{p,q}_\varphi ({\mathbb R}^n)= M^{p,q}_\psi({\mathbb R}^n)$ and $W^{p,q}_\varphi ({\mathbb R}^n)= W^{p,q}_\psi({\mathbb R}^n)$
for all $\varphi,\psi \in {\mathcal S}({\mathbb  R}^n) \setminus \{ 0 \}$
([\ref{Feichtinger}, Theorem 6.1]).
This property is crucial in the sequel,
since we choose a suitable window function $\varphi$
to estimate the modulation space norm.
If $1 \leq p,q < \infty$ then
${\mathcal S}({\mathbb R}^n)$ is dense
in $M^{p,q}$  and $W^{p,q}$([\ref{Feichtinger}, Theorem 6.1]).
We also note $L^2 =M^{2,2}=W^{2,2}$,
and $M^{p_1,q_1} \hookrightarrow M^{p_2,q_2}$
if $p_1 \leq p_2, q_1 \leq q_2$([\ref{Feichtinger}, Proposition 6.5]) and 
$W^{p,q} \hookrightarrow M^{p,q}$ ($M^{p,q} \hookrightarrow W^{p,q}$) if $q\ge p$ (resp. $p\ge q$). 

Let us define by ${\mathcal M}^{p,q}({\mathbb R}^n)$ and ${\mathcal W}^{p,q}({\mathbb R}^n)$
the completion of ${\mathcal S}({\mathbb R}^n)$
under the norm $\| \cdot \|_{M^{p,q}}$ and the norm $\| \cdot \|_{W^{p,q}}$ respectively. 
Then ${\mathcal M}^{p,q}({\mathbb R}^n) = M^{p,q}({\mathbb R}^n)$ and 
${\mathcal W}^{p,q}({\mathbb R}^n) = W^{p,q}({\mathbb R}^n)$
for $1 \leq p,q < \infty$.
\section{Key Lemmas}
The orbit of the classical mechanics corresponding to \eqref{SE}
 is described by the system of ordinary differential equations
    \begin{align}
    \label{ODE}
    \begin{cases}
    \dot{x}(s)=\xi (s),\\
    \dot{\xi} (s)=-(\nabla _x V)(s,x(s)),
    \end{cases}
    \end{align}
 where $x:\re\rightarrow \re^n$ and $\xi :\re\rightarrow \re^n$ (see also 
 Fujiwara \cite{Fujiwara}).
We denote $x(s;t,x,\xi ), \xi (s;t,x,\xi )$ by the pair of solutions to \eqref{ODE} 
with $x(t)=x, \xi (t)=\xi$. 
\begin{lemma}
\label{lemma-1}
Under Assumption \ref{ass1-1} on  $V(t,x)$, there 
exists $T>0$ such that 
$$
\frac{1}{2}\le\left|\frac{\partial (x(t;0,x,\xi ))}{\partial(x)}\right|
\le 2
$$
uniformly with respect to $(x,\xi )$ in $\re^n\times \re^n$ for $|t|\le T$, 
where $\frac{\partial (x(t;0,x,\xi ))}{\partial(x)}$ is 
the Jacobian of $x(t;0,x,\xi )$ as a function 
$x\mapsto x(t;0,x,\xi )$. 
\end{lemma}
\begin{proof}
We write $x(t), \xi(t)$ as $x(t;0,x,\xi ), \xi (t;0,x,\xi)$ for abreviation. 
Since $x(t), \xi(t)$ is the pair of the solutions to \eqref{ODE}, we have 
$$
x(t)= x + \xi t - \int_{0}^{t}(t-\tau )\nabla V(\tau, x(\tau ))d\tau, 
$$
from which we obtain 
\begin{equation}
\label{eq-lemma-4}
 \frac{\partial x_j(t)}{\partial x_k}= \delta_{jk} - \sum_{l=1}^{n}
\int_{0}^{t}(t-\tau )\partial_{x_l}\partial_{x_j} V(\tau, x(\tau )) 
\frac{\partial x_l(\tau )}{\partial x_k}
d\tau. 
\end{equation}
Since $\partial_{x_k}\partial_{x_l} V(\tau, x(\tau )) $ is bounded with respect to $(x,\xi )$, 
we have 
\begin{equation}
\label{eq-lemma-5}
 \sup_{|t|\le T} \left|\frac{\partial x_j(t)}{\partial x_k}\right|
\le \delta_{jk} +
Cn \left|\int_{0}^{t}|t-\tau|
\sup_{|t|\le T, 1\le l\le n}\left|\frac{\partial x_l(\tau )}{\partial x_k}\right|
d\tau\right|
\le 1 +Cnt^2\sup_{|\tau |\le T, 1\le l\le n}\left|\frac{\partial x_l(\tau )}{\partial x_k}\right|, 
\end{equation}
where $\delta_{jk}$ stands for the Kronecker delta, 
from which we get 
\begin{equation}
\label{eq-lemma-6}
 \sup_{|t|\le T, 1\le j,k\le n}\left|\frac{\partial x_j(\tau )}{\partial x_k}\right|
\le 2 
\end{equation}
for $|t|\le T$ with some $T>0$. 
Taking $t\to 0$ for the both sides of \eqref{eq-lemma-4}, we have
$$
\frac{\partial x_j(t)}{\partial x_k} \to \delta_{jk} 
$$ 
uniformly with repect to $(x,\xi )$ in $\re^n\times \re^n$. 
Hence we have 
$$
\frac{\partial (x(t;0,x,\xi ))}{\partial(x)}\to 1 
$$
as $t\to 0$ uniformly in $(x,\xi)$, which completes the proof. 
\end{proof}
Changing the variables $(x,\xi )\to (X, \xi)$ with $X = x(t;0,x,\xi)$, 
the variable $x$ is a function of $(X,\xi )$ satisfying 
\begin{equation}
 X = x(X,\xi ) +  \xi t -\int_{0}^{t}(t-\tau )\nabla V(\tau, x(\tau;0,x(X,\xi ) ,\xi))d\tau, 
\end{equation}
from which we have the following lemma. 
\begin{lemma}
\label{lemma-2}
Under Assumption \ref{ass1-1} on  $V(t,x)$, 
there exists $T>0$ such that 
$$
\frac{1}{2}\le\left|\frac{\partial (\xi (t;0,x(X,\xi),\xi ))}{\partial(\xi)}\right|
\le 2
$$
uniformly with respect to $(x,\xi )\in \re^n\times \re^n$ for $|t|\le T$. 
\end{lemma}
The proof of Lemma \ref{lemma-2} can be done by the same argument as in the proof of Lemma \ref{lemma-1}. 

\section{Proof of Theorem \ref{main-estimate}}
For the proof, we use the transformed equation of \eqref{SE} by wave packet transform, 
which is obtained in \cite{K-K-I-4}. 
Transforming the initial value problem \eqref{SE} by the wave packet transform 
with $\varphi(t,x)=e^{\frac{i}{2}t\triangle}\varphi_0$ as a basic wave packet, we have 
    \begin{multline}
    \label{kainohyouji}
    W_{\varphi (t,\cdot)}u(t,x,\xi)
    =e^{-i\int_0^t h(s;t,x,\xi)ds}
    \bigg( W_{\varphi_0} u_0(x(0;t,x,\xi),\xi(0;t,x,\xi))\\
    -i\int_0^t e^{i\int_0^\tau h(s;t,x,\xi)ds}Ru(\tau ,x(\tau;t,x,\xi), \xi
     (\tau;t,x,\xi)) d\tau\bigg),
    \end{multline}
where $h(s;t,x,\xi)=\frac{1}{2}|\xi (s;t,x,\xi)|^2+
    V(s,x(s;t,x,\xi))-\nabla_x V(s,x(s;t,x,\xi))\cdot x(s;t,x,\xi)$ and 
    \begin{multline*}
    Ru(t,x,\xi)
    =
    \frac{1}{\|\varphi(t,\cdot)\|^2_{L^2}}
    \sum_{j,k=1}^n \iiint_{\re^{3n}} \varphi_{jk}(t,y-x) V_{jk}
    (t,x,y) \varphi(t,y-z) \\
    \times W_{\varphi(t,\cdot)}u(t,z,\eta) 
    e^{iy\cdot(\eta -\xi)} dz\dbar\eta dy,
    \end{multline*}
 where $\varphi_{jk}(t,y)=y_j y_k \overline{\varphi (t,y)}$ and 
    $
    V_{jk}(t,x,y)=\int_0^1 \partial_{x_j}\partial_{x_k}
    V(t,x+\theta(y-x))(1-\theta)d\theta 
    $. 
\begin{proof}[Proof of Theorem \ref{main-estimate}]
 We only show in the case that $t>0$.
 Taking $L^p$ norm with respect to $x$ and $L^q$ norm with respect to $\xi$ of the 
 both sides of \eqref{kainohyouji}, 
 we have 
 \begin{multline}
\label{trans-eq}
 \Vert u(t,\cdot )\Vert_{M^{p, q}_\varphi(t)}\\
 \le 
 \Vert W_{\varphi_0} u_0(x(0;t,x,\xi),\xi(0;t,x,\xi))\Vert_{L^p_xL^q_\xi} 
 + \int_0^t \Vert Ru(\tau ,x(\tau;t,x,\xi), \xi(\tau;t,x,\xi))\Vert_{L^p_xL^q_\xi}  d\tau
 \end{multline}
Next we estimate the first term of the right hand side of \eqref{trans-eq}. 
Changing variable $x$ to $X = x(0;t,x,\xi )$, we have with Lemma \ref{lemma-1}
and Minkowski's inequality 
\begin{align*}
& \Vert W_{\varphi_0} u_0(x(0;t,x,\xi),\xi(0;t,x,\xi))\Vert_{L^p_xL^q_\xi} \\
&=  \Vert \Vert W_{\varphi_0} u_0(x(0;t,x,\xi),\xi(0;t,x,\xi))\Vert_{L^p_x}\Vert_{L^q_\xi} \\
&= \Vert \Vert W_{\varphi_0} u_0(X,\xi(0;t,x,\xi))\frac{\partial (x)}{\partial (X)}
\Vert_{L^p_x}\Vert_{L^q_\xi} \\
&\le C \Vert \Vert W_{\varphi_0} u_0(X,\xi(0;t,x,\xi))
\Vert_{L^p_X}\Vert_{L^q_\xi}\\
&= C \left\Vert 
\int\left| W_{\varphi_0} u_0(X,\xi(0;t,x,\xi))\right|^pdX
\right\Vert^{1/p}_{L^{q/p}_\xi}\\
&\le C
\left(
\int\left\Vert  |W_{\varphi_0} u_0(X,\xi(0;t,x,\xi))|^p\right\Vert_{L^{q/p}_\xi}dX
\right)^{1/p}\\
& = C \Vert \Vert W_{\varphi_0} u_0(X,\xi(0;t,x,\xi))
\Vert_{L^q_\xi}\Vert_{L^p_X}. 
\end{align*}
Changing variable $\xi(0;t,x(X,\xi),\xi)$ to $\Xi$ with Lemma \ref{lemma-2}, 
we have 
\begin{align}
 & \Vert W_{\varphi_0} u_0(x(0;t,x,\xi),\xi(0;t,x,\xi))\Vert_{L^p_xL^q_\xi} \nonumber\\
& = C \left\Vert \left(\int |W_{\varphi_0} u_0(X,\xi(0;t,x(X,\xi),\xi))|^q d\Xi \right)^{1/q}
\right\Vert_{L^p_X}\nonumber\\
& = C \left\Vert \left(\int |W_{\varphi_0} u_0(X,\Xi )|^q \frac{\partial \xi(0;t,x(X,\xi),\xi)}{\partial \Xi }d\Xi \right)^{1/q}
\right\Vert_{L^p_X}\nonumber\\
& \le C \left\Vert \left(\int |W_{\varphi_0} u_0(X,\Xi )|^q d\Xi \right)^{1/q}
\right\Vert_{L^p_X}\nonumber\\
& = C \Vert  u_0 \Vert_{W^{p,q}_{\varphi_0}}. \label{eq-13}
\end{align}
Thirdly we estimate the second term of the right hand side of \eqref{trans-eq}. 
The same argument as in the second part of this proof yields 
\begin{align}
& \Vert Ru(\tau ,x(\tau;t,x,\xi), \xi(\tau;t,x,\xi))\Vert_{L^p_xL^q_\xi} \\
& \le C \Vert Ru(\tau ,X, \xi(\tau;t,x,\xi))\Vert_{L^q_\xi L^p_X} \\
& \le C \Vert Ru(\tau ,X, \Xi )\Vert_{L^q_\Xi L^p_X}. 
\end{align}
The fact that $(1-\triangle_\xi )e^{i\xi(y-x)}=(1+|y-x|^2)e^{i\xi(y-x)}$ and integration by parts yield that 
\begin{align}
& {\lVert{ \lVert Ru(\tau,x,\xi ) \lVert}_{L_\xi^q}\lVert}_ {L_x^p}   \nonumber \\
&\leq \frac{1}{{\lVert \varphi(\tau,\cdot) \lVert}_{L^2}^2} \sum_{j,k=1}^{n}  \bigg{\lVert} \bigg{ \lVert}  \int \int \int | {(1-\Delta )}^{N}\varphi_{jk}(\tau,y-x) \nonumber   \\
&\phantom{aaaaaaaaaaaaaaa} \times V_{jk}(\tau,x,y) \nonumber  \varphi(\tau,y-z)| \frac{|W_{\varphi(\tau,\cdot)}u(\tau,z,\eta)|}{{\langle \eta-\xi \rangle}^{2N}} dz d\bar{\eta} dy\ \bigg{\Vert}_ {L_x^p}  \bigg{\lVert}_{L_\xi^q} \nonumber \\
 &\leq C \frac{1}{{\lVert \varphi(\tau,\cdot) \lVert}_{L^2}^2} \sum_{j,k=1}^{n} \sum_{|\beta_1|+|\beta_2|+|\beta_3|\leq2N}  \bigg{\lVert}  \bigg{\lVert}  \int \int \int |\partial_y^{{\beta}_1} \varphi_{jk}(\tau,y-x)\\
 &\phantom{aaaaaaaaaaaaaaa} \times \partial_y^{{\beta}_2} V_{jk}(\tau,x,y) \nonumber  \partial_y^{\beta_3}\varphi(\tau,y-z)| \frac{|W_{\varphi(\tau,\cdot)}u(\tau,z,\eta)|}{{\langle \eta-\xi\rangle}^{2N}} dz d\bar{\eta} dy\bigg{\lVert}_ {L_x^p}  \bigg{\lVert}_{L_\xi^q}. \label{170724_23Feb24} 
\end{align}
Since 
$(1+|A|)^\epsilon \le (1+|A+B|)^\epsilon (1+|B|)^\epsilon$, we have 
$(1+|A+B|)^{-\epsilon} \le (1+|A|)^{-\epsilon} (1+|B|)^\epsilon$, 
from which we have
\begin{align*}
 |V_{jk}(t,x,y)|&\le 
\int_0^1 |\partial_{x_j}\partial_{x_k}V(t,x+\theta(y-x))(1-\theta)|d\theta \\
&\le \int_0^1 C (1+|x+\theta(y-x)| )^{-\epsilon}d\theta \\
&\le  C (1+|x| )^{-\epsilon}(1+|y-x|)^\epsilon .
\end{align*}
The above inequality and \eqref{170724_23Feb24} show that 
\begin{align}
& {\lVert{ \lVert Ru(\tau,x,\xi ) \lVert}_{L_\xi^q}\lVert}_ {L_x^p}   \nonumber \\
&\leq C \frac{1}{{\lVert \varphi(\tau,\cdot) \lVert}_{L^2}^2} \sum_{j,k=1}^n \sum_{|\beta_1|+|\beta_2|+|\beta_3|\leq2N} \bigg{\lVert} \bigg{\lVert}  \int \int \int {\langle y-x \rangle}^{{\epsilon}}|\partial_y^{{\beta}_1} \varphi_{jk}(\tau,y-x)|\nonumber\\
&\phantom{aaaaa} \times
| \partial_y^{\beta_3}\varphi(\tau,y-z)| \frac{|W_{\varphi(\tau,\cdot)}u(\tau,z,\eta)|}{{\langle \eta-\xi\rangle}^{2N}} {\langle x \rangle}^{{-\epsilon}}dz d\bar{\eta} dy\bigg{\lVert}_ {L_x^p}  \bigg{\lVert}_{L_\xi^q}. 
\end{align}
Hausdorff-Young's inequality, H\"older's inequality and Minkowski's inequality
 show with $r>1, 1/r+1/r'=1/p, \epsilon r>n$ 
that
\begin{align*}
& {\lVert{ \lVert Ru(\tau,x,\xi ) \lVert}_{L_\xi^q}\lVert}_ {L_x^p}  \nonumber \\
&\lesssim \sum_{j,k=1}^n \sum_{|\beta_1|+|\beta_2|+|\beta_3|\leq2N}   {\lVert {\langle \cdot \rangle}^{-\epsilon} \lVert}_{L^r} 
\lVert \langle \cdot \rangle^{\epsilon} \partial^{\beta_1}\varphi_{j,k} \lVert_{L^1} \lVert \partial^{\beta_2}\varphi \lVert_{L^1}
\lVert \langle \cdot \rangle^{-2N} \lVert_{L^1 }
{\lVert{ \lVert W_{\varphi(\tau,\cdot)}u(X,\Xi)  \lVert}_ {L_\Xi^q}  \lVert}_{L_X^{r{\prime}}} \\
&\vphantom{\sum}\leq \sum_{j,k=1}^n \sum_{|\beta_1|+|\beta_2|+|\beta_3|\leq2N}   {\lVert {\langle \cdot \rangle}^{-\epsilon} \lVert}_{L^r} \lVert \langle \cdot \rangle^{\epsilon} \partial^{\beta_1} \varphi_{j,k} \lVert_{L^1} \lVert \partial^{\beta_2}\varphi \lVert_{L^1}
\lVert \langle \cdot \rangle^{-2N} \lVert_{L^1 }
{\lVert{ \lVert  W_{\varphi(\tau,\cdot)}u(\tau,\cdot)  \lVert}_ {L_X^{r^{\prime}}}  \lVert}_{L_\Xi^q}, 
\end{align*}
which shows that 
\begin{equation}
\label{eq-18}
 {\lVert{ \lVert Ru(\tau,x(\tau ;t,x,\xi ),\xi(\tau ;t,x,\xi) ) \lVert}_ {L_x^p}  \lVert}_{L_\xi^q} 
\le C{\lVert{ \lVert  W_{\varphi(\tau,\cdot)}u(\tau,\cdot)  \lVert}_ {L_X^{r'}}  \lVert}_{L_\Xi^q} 
\le C' {\lVert{ \lVert  W_{\varphi(\tau,\cdot)}u(\tau,\cdot)  \lVert}_ {L_X^p}  \lVert}_{L_\Xi^q} .
\end{equation}
From \eqref{eq-13} and \eqref{eq-18}, we have 
 \begin{equation*}
 \Vert u(t,\cdot )\Vert_{M_{\varphi(t)}^{p,q}}
 \le C_1 \Vert W_{\varphi_0} u_0(x,\Xi )\Vert_{W^{p,q}_{\varphi_0}} 
 + C_2 \int_0^t \Vert u(\tau ,\cdot )\Vert_{M_{\varphi(\tau )}^{p,q}}  d\tau, 
 \end{equation*}
which and Gronwall inequality complete the proof. 
\end{proof}


\address{
Kosuzu Hamaoka\\
Department of Mathematics,\\
 Faculty of Science, \\
Tokyo University of Science,\\
Kagurazaka 1-3, Shinjuku-ku, \\
Tokyo 162-8601, Japan}
{1122520@ed.tus.ac.jp}

\address{
Keiichi Kato\\
Department of Mathematics,\\
 Faculty of Science, \\
Tokyo University of Science,\\
Kagurazaka 1-3, Shinjuku-ku, \\
Tokyo 162-8601, Japan}
{kato@ma.kagu.tus.ac.jp}

\address{
Shun Takizawa\\
Department of Mathematics,\\
 Faculty of Science, \\
Tokyo University of Science,\\
Kagurazaka 1-3, Shinjuku-ku, \\
Tokyo 162-8601, Japan}
{1123703@ed.tus.ac.jp}



\end{document}